\documentclass[12pt]{amsart}

\usepackage{amssymb}

\usepackage{enumerate}

\makeatletter
\@namedef{subjclassname@2010}{%
  \textup{2010} Mathematics Subject Classification}
\makeatother

\newtheorem{thm}{Theorem}[section]
\newtheorem{cor}[thm]{Corollary}
\newtheorem{lem}[thm]{Lemma}

\theoremstyle{definition}
\newtheorem{defin}[thm]{Definition}

\numberwithin{equation}{section}

\usepackage{mathrsfs}

\newcommand{\zfc}{\mathnormal{\mathsf{ZFC}}}
\newcommand{\ch}{\mathnormal{\mathsf{CH}}}
\newcommand{\fu}{\mathop{\mathrm{FU}}}
\newcommand{\cpa}{\mathnormal{\mathsf{CPA}}}
\newcommand{\homh}{\mathnormal{\mathfrak{hom}_H}}
\newcommand{\covm}{\mathnormal{\mathrm{cov}(\mathcal M)}}

\DeclareMathOperator{\ran}{ran}

\DeclareMathOperator{\cov}{cov}

\begin{document}


\baselineskip=17pt



\title[Union ultrafilters]{A parametrized diamond principle and union ultrafilters}

\author[D. Fern\'andez]{David Fern\'andez-Bret\'on}
\address{
Department of Mathematics\\
University of Michigan\\
2074 East Hall, 530 Church Street \\
Ann Arbor, MI 48109-1043, U.S.A.}
\email{djfernan@umich.edu}
\urladdr{http://www-personal.umich.edu/\textasciitilde djfernan/}

\author[M. Hru\v{s}\'ak]{Michael Hru\v{s}\'ak}
\address{
Centro de Ciencias Matem\'aticas\\
UNAM--Campus Morelia\\ 
C.P. 58190\\
Morelia, Michoac\'an\\
M\'exico}
\email{michael@matmor.unam.mx}
\urladdr{http://www.matmor.unam.mx/\textasciitilde michael/}

\date{}

\begin{abstract}
We consider a cardinal invariant closely related to Hindman's theorem. We prove that this cardinal invariant is small in the iterated Sacks perfect set forcing model, and that its corresponding parametrized diamond principle implies the existence of union ultrafilters. 
As a corollary, this establishes the existence of union ultrafilters in the iterated Sacks model of Set Theory.
\end{abstract}



\maketitle

\section{Introduction}

Union ultrafilters are ultrafilters that relate naturally to Hindman's theorem (in much the same way that Ramsey ultrafilters relate naturally to Ramsey's theorem). Recall that Hindman's theorem~\cite[Theorem 3.1]{hindmanthm}, in one of its equivalent forms (the equivalence, which is explicitly pointed out in~\cite[p. 384]{baumgartner}, can be seen via the identification of each natural number with the finite set of places where its binary expansion has a nonzero digit), states that whenever we partition the set $[\omega]^{<\omega}$, consisting of all finite subsets of $\omega$, in two cells, $[\omega]^{<\omega}=A_0\cup A_1$, there exists an infinite pairwise disjoint family $X\subseteq[\omega]^{<\omega}$ such that the set
\begin{equation*}
\fu(X)=\left\{\bigcup_{x\in F}x\bigg|F\in[X]^{<\omega}\setminus\{\varnothing\}\right\}
\end{equation*}
is contained in one single $A_i$. Union ultrafilters were first introduced by Blass~\cite[p. 92]{blass-union}\footnote{A close relative of union ultrafilters, known as \textit{strongly summable ultrafilters}, were introduced at around the same time by Hindman~\cite[Definition 2.1]{hindman-obscure}, who had actually already inadvertently proved that their existence follows from the continuum hypothesis $\ch$ earlier, in~\cite[Theorem 3.3]{hindman-existence}. On the other hand, Matet~\cite{matet} arrived independently to an equivalent concept, in terms of what he called \textit{filters of partitions}. Specifically, he established several facts about objects that he calls \textit{Hindman filters} (union ultrafilters) and \textit{Milliken-Taylor filters} (stable ordered union ultrafilters).}, and have been studied by several people afterwards, including the first author of this paper~\cite{blass-near-coherent,blasshindman,eisworth,tesis,yo-canadian,jurisetal,krautz-union}.

\begin{defin}
An ultrafilter $u$ on $[\omega]^{<\omega}$ is said to be a \textbf{union ultrafilter} if, for every $A\in u$, there exists an infinite, pairwise disjoint family $X\subseteq[\omega]^{<\omega}$ such that $A\subseteq\fu(X)\ni u$.
\end{defin}

One of the most interesting questions regarding union ultrafilters, is that of their existence. The existence of union ultrafilters\footnote{And, consequently, also the existence of strongly summable ultrafilters, by~\cite[Theorem 1]{blasshindman}.} follows from the Continuum Hypothesis, abbreviated $\ch$ (i.e. the statement that the set of real numbers is in bijection with the first uncountable cardinal, or $|\mathbb R|=\omega_1$) by~\cite[Theorem 3.3]{hindman-existence}. On the other hand, Blass and Hindman proved~\cite[Theorem 2]{blasshindman} that the existence of a union ultrafilter implies the existence of a P-point, and hence cannot be established within the $\zfc$ axiomatic framework, by Shelah's result on the independence of the existence of P-points from the $\zfc$ axioms\footnote{Quite recently, a much simpler proof of the consistency of no P-points has been found by Chodounsk\'y and Guzm\'an~\cite{david-osvaldo}, who proved that there are no P-points in the iterated nor in the side-by-side Silver model.}~\cite[Theorem 6.5]{wimmers}. These existential results turn out to be relevant to the study of extremally disconnected topological groups, since union ultrafilters naturally give rise to such groups. Specifically, if $u$ is a union ultrafilter, then declaring $\{A\cup\{\varnothing\}\big|A\in u\}$ to be the neighborhood filter of the identity element $\varnothing$ yields an extremally disconnected (even maximal) group topology on the Boolean group $[\omega]^{<\omega}$, and hence provides consistent examples of these kinds of groups~\cite{arkhangelskii,malykhin,protasov,sipacheva,ariet}.

From the set-theoretic viewpoint, it is relevant to determine which, from among the many models of Set Theory that can be obtained by forcing, satisfy the statement asserting the existence of union ultrafilters. For example, Eisworth~\cite[Theorem 9]{eisworth} established that the existence of union ultrafilters follows from the cardinal invariant equality $\covm=\mathfrak c$ (which is equivalent to Martin's Axiom restricted to countable forcing notions), and consequently all models satisfying such equality (this includes Cohen's classical model, as well as all models obtained by finite support iterations of ccc forcing notions whose length equals $\mathfrak c$ in the final extension) will satisfy such an existential statement. The first author has constructed~\cite[Section 4.3]{tesis} models of $\zfc$ that contain union ultrafilters, even though $\covm<\mathfrak c$ in those models. On the other hand, there are also results establishing that the existence of union ultrafilters implies the existence of other set-theoretic objects such as semiselective ultrafilters~\cite[Proposition 7.5]{matet}, or pairs of non-near-coherent ultrafilters~\cite[Theorem 38]{blass-near-coherent}, which allow us to conclude that many of the better-studied \textit{canonical models}\footnote{Here, a \textit{canonical model} is a model of $\zfc$ obtained by means of a countable support iteration of proper Borel forcing notions.} do not satisfy the existence of union ultrafilters: Miller's model because it satisfies the Near-Coherence of Filters principle~\cite[Theorem, p. 530]{blass-ncf}, Laver's and Mathias's models because in these models, rapid ultrafilters do not exist~\cite[Theorem 5; Remark (1) on p. 106]{miller}, or the iterated Silver model, because by the recent result of~\cite{david-osvaldo}, this model does not have P-points. Another model (though not canonical) that has been studied extensively is the Random model, where union ultrafilters cannot exist either, because this model does not have semiselective ultrafilters~\cite[Theorem 5.1]{kunen} (that is, rapid P-points).

The very interesting question of whether the iterated Sacks model (obtained by a countable support iteration of perfect set forcings) satisfies such an existential statement is not immediately answered with these sorts of considerations, and was actually explicitly asked by the first author~\cite[Question 4.11, p. 106]{tesis}. We have been able to answer this question in the affirmative by using the theory of parametrized diamond principles, applied to a particular cardinal characteristic of the continuum, and this is precisely the main result of this paper. It should be noted that a result of Zheng~\cite[Theorem 0.3]{y2zheng}, obtained independently and at around the same time, provides an alternative proof for the main result of this paper.\footnote{Strictly speaking, Zheng proves that there are union ultrafilters in the \textit{side-by-side} Sacks model, although it seems plausible that her argument might be suitably modified to ensure the existence of these ultrafilters in the iterated Sacks model too. In any case, the proof in the present paper is substantially different from that of Zheng's, and arguably interesting in its own right.}

In Section 2 we define the cardinal invariant that we will use, along with its parametrized diamond principle, and show that said cardinal invariant is small in the iterated Sacks model, which in turn implies that the corresponding diamond principle holds in the iterated Sacks model. Then in Section 3 we prove that the aforementioned parametrized diamond principle implies the existence of a union ultrafilter.

\section{A cardinal invariant, its parametrized diamond principle, and the iterated Sacks model}

In this section we will start by considering the following cardinal characteristic of the continuum, whose definition arises directly from Hindman's theorem. Our notation for this invariant follows that of~\cite{blasscardinv}.

\begin{defin}
 The cardinal invariant $\homh$ is defined to be the least cardinality of a family $\mathcal X$ 
 of infinite ordered\footnote{A family $X\subseteq[\omega]^{<\omega}$ is said to be \textit{ordered} if, for every $x,y\in X$, either $\max(x)<\min(y)$ or $\max(y)<\min(x)$. Ordered families are sometimes called \textit{block sequences}. Throughout this paper, we will use both terminologies indistinctly.} 
 subfamilies of $[\omega]^{<\omega}$ such that for every colouring 
 of $[\omega]^{<\omega}$ into finitely many colours, there exists an $X\in\mathcal X$ such 
 that $\fu(X)$ is monochromatic.
\end{defin}

Since within every set of the form $\fu(X)$ we can find infinite ordered subfamilies, in the 
above definition we could have dropped the requirement that every element of $\mathcal X$ be ordered, 
and we would still have obtained an equivalent cardinal. In fact, by using symmetric differences 
instead of unions (note that in the case of a pairwise disjoint family the two approaches are exactly the 
same), we could even have dropped the requirement that every element of $\mathcal X$ consists 
of pairwise disjoint sets, and we would still have obtained the exact same cardinal invariant.

In~\cite{weakdiamond}, a general theory of weak diamond principles is studied. 
The cardinal invariants considered there are defined in terms of triples $(A,B,E)$ 
where $|A|\leq\mathfrak c$, 
$|B|\leq\mathfrak c$ and $E\subseteq A\times B$, satisfying that 
$(\forall a\in A)(\exists b\in B)(a\ E\ b)$ and that 
$(\forall b\in B)(\exists a\in A)\neg(a\ E\ b)$ (the reason for the 
last two requirements is just so we can ensure the existence and nontriviality 
of the corresponding 
cardinal invariant). The cardinal invariant associated to such a triple 
$(A,B,E)$ (sometimes referred to as the \textit{evaluation} of the triple), 
which we denote by $\langle A,B,E\rangle$, is given by
\begin{equation*}
\langle A,B,E\rangle=\min\{|X|\big|X\subseteq B\wedge(\forall a\in A)(\exists b\in X)(a\ E\ b)\}.
\end{equation*}

We are especially interested in cardinal invariants  $\langle A,B,E\rangle$ that are Borel, 
i.e. all entries of the triple $(A,B,E)$, can be viewed as Borel subsets of some Polish space. For 
a cardinal invariant of this type, the \emph{parametrized diamond principle} of the triple is the statement 
that for every Borel function 
$F:2^{<\omega_1}\longrightarrow A$ (where Borel just means that for every $\alpha<\omega_1$ 
the restriction $F\upharpoonright2^{<\alpha}$ is a Borel function) there exists a 
$g:\omega_1\longrightarrow B$ such that for every $f:\omega_1\longrightarrow 2$, the set
\begin{equation*}
\{\alpha<\omega_1\big|F(f\upharpoonright\alpha)\ E\ g(\alpha)\}
\end{equation*}
is stationary. This 
statement is denoted by $\diamondsuit(A,B,E)$. The fundamental theorem regarding these 
parametrized diamond principles is the following:

\begin{thm}[\cite{weakdiamond}, Theorem 6.6]\label{fundamental}
 Let $(A,B,E)$ be a Borel triple defining a cardinal invariant. Let 
 $\langle\mathring{\mathbb Q_\alpha}\big|\alphaÂ­<\omega_2\rangle$ be a sequence 
 of (names for) Borel forcing notions that satisfy 
$\Vdash_\alpha``\mathring{\mathbb Q_\alpha}\cong\check{2^+}\times\mathring{\mathbb Q_\alpha}"$ (here $2^+$ denotes the partial order that consists of two incomparable elements, plus one maximal element $\mathbb 1$ above them), 
and assume that the countable support iteration 
 $\mathbb P_{\omega_2}$ of these forcing notions is proper. Then 
 \begin{equation*}
  \mathbb P_{\omega_2}\Vdash``\diamondsuit(A,B,E)"\text{ if and only if 
 }\mathbb P_{\omega_2}\Vdash``\langle A,B,E\rangle\leq\omega_1".
 \end{equation*}
\end{thm}

All of the above is relevant to our cardinal invariant $\homh$ mainly because its definition can be thought of as given by
\begin{equation*}
\homh=\langle\wp([\omega]^{<\omega}),\{X\in\wp([\omega]^{<\omega})\big|X\text{ is a block sequence}\},\text{is }\fu\text{-reaped by}\rangle
\end{equation*}
where a set $A$ is $\fu$-reaped by a block sequence $X$ if either $\fu(X)\subseteq A$, or $\fu(X)\cap A=\varnothing$ 
(i.e. if $\fu(X)$ is homogeneous for the colouring $\{A,[\omega]^{<\omega}\setminus A\}$). So 
in this context, the combinatorial principle 
$\diamondsuit(\homh)$ is the statement that for every function 
$F:2^{<\omega_1}\longrightarrow\wp([\omega]^{<\omega})$, there exists an 
$\omega_1$-sequence $g$ of block sequences such that, for every $f:\omega_1\longrightarrow2$, 
$F(f\upharpoonright\alpha)$ is $\fu$-reaped by $g(\alpha)$  (i.e. $\fu(g(\alpha))$ reaps 
$F(f\upharpoonright\alpha)$) for stationarily many $\alpha<\omega_1$. And in order to see that such a principle holds in a forcing extension (such as the one obtained by iterating  Sacks forcing) obtained by a countable support iteration of proper forcing notions, it will suffice to establish that in such extension the cardinal $\homh$ equals $\omega_1$. So our next immediate goal is to establish that $\homh=\omega_1$ in the iterated Sacks model, and in order to attain this goal, we now proceed to state some definitions and results from~\cite{forcingidealized}, which will be used for this purpose.

\begin{defin}[\cite{forcingidealized}, Def. 6.1.9]
A cardinal invariant $\mathfrak x$ is said to be \textbf{very tame} if its definition is of the form
\begin{equation*}
\mathfrak x=\min\{|X|\big|\phi(X)\wedge(\forall x\in\mathbb R)(\exists y\in X)(\theta(x,y))\},
\end{equation*}
where
\begin{itemize}
\item $\phi(X)$ is a formula of the form 
\begin{equation*}
(\forall x_0,\ldots,x_k\in X)(\exists y_0,\ldots,y_l\in X)(\psi(x_0,\ldots,x_k,y_0,\ldots,y_l))
\end{equation*}
for some arithmetic formula\footnote{An arithmetic formula is one all of whose quantifiers are bounded by $\omega$, or in other words, all of its quantifiers range only over natural numbers.} $\psi$;
\item $\theta(x,y)$ is an analytic\footnote{An analytic formula is one that consists of a quantifier ranging over real numbers (or over subsets of $\omega$, or any other equivalent object) followed by an arithmetic formula.} formula;
\item $\zfc\vdash(\forall X\subseteq\mathbb R)((|X|=\omega\wedge\phi(X))\Rightarrow(\exists Z)(X\subseteq Z\wedge|Z|=\mathfrak x\wedge Z\text{ witnesses }\mathfrak x))$.
\end{itemize}
\end{defin}

\begin{lem}\label{homhistame}
$\homh$ is very tame.
\end{lem}

\begin{proof}
We start by canonically identifying $\omega$ with $[\omega]^{­<\omega}$ (during the proof, at some point we will need such an identification to be expressible by a $\Delta_0^0$ formula, which can be done, for example, by means of binary expansions), so that we have a corresponding coding of $\wp([\omega]^{<\omega})$ by means of $\mathbb R$. Thus, we just need to find a formula $\phi(\mathscr X)$ of the form $(\forall x_0,\ldots,x_k\in\mathscr X)(\exists y_0,\ldots,y_l\in\mathscr X)(\psi(x_0,\ldots,x_k,y_0,\ldots,y_l))$ for some arithmetic $\psi$; as well as an analytic formula $\theta$, such that
\begin{equation*}
\homh=\min\{|\mathscr X|\big|\phi(\mathscr X)\wedge(\forall A\in\wp([\omega]^{­<\omega}))(\exists X\in\mathscr X)(\theta(A,X))\},
\end{equation*}
and
\begin{align*}
\zfc\vdash & (\forall\mathscr X\subseteq\wp([\omega]^{<\omega}))((|\mathscr X|=\omega\wedge\phi(\mathscr X))\Rightarrow \\
 & (\exists\mathscr Z)(\mathscr X\subseteq\mathscr Z\wedge|\mathscr Z|=\homh\wedge\mathscr Z\text{ witnesses }\homh)).
\end{align*}

Since $\homh$ is given by
\begin{align*}
\homh=\min\{|\mathscr X|\big| & \mathscr X\text{ consists of infinite block sequences}\wedge \\ 
 & (\forall A\in\wp([\omega]^{<\omega}))(\exists X\in\mathscr X)(X\text{ }\fu\text{-reaps }A)\},
\end{align*}
it suffices to show that the formula $\phi(\mathscr X)$ expressing that $\mathscr X$ is a family of infinite block sequences, as well as the formula $\theta(A,X)$ expressing that $A$ is $\fu$-reaped by $X$, are as required.

We start by looking at the formula $\psi(X)$ which expresses that $X$ is an infinite block sequence. Such a formula, when written down explicitly, reads as follows:
\begin{align*}
 & (\forall x\in X)(\forall y\in X)(\max(x)<\min(y)\vee\max(y)<\min(x)) \\
 \wedge & (\forall x\in X)(\exists y\in X)(\max(x)<\min(y)).
\end{align*}
Since the quantifiers of this formula range only over elements of $[\omega]^{­<\omega}$ (which are just natural numbers according to our coding), the formula is arithmetic. Now, the formula $\phi(\mathscr X)$ expressing that $\mathscr X$ is a family of infinite block sequences is simply given by $(\forall X\in\mathscr X)(\psi(X))$; so this allows us to conclude that $\phi(\mathscr X)$ is as required.

Secondly, the formula $\theta(A,X)$ must be given by
\begin{equation*}
(\forall F\in[X]^{<\omega})(\bigcup F\in A)\vee(\forall F\in[X]^{<\omega})(\bigcup F\notin A),
\end{equation*}
which is clearly analytic, as it contains only one quantifier (and the definition of $\bigcup F$ can be carried out by means of a $\Delta_0^0$ formula).

Finally, it is trivial that every countable family $\mathscr X$ of infinite block sequences can be extended to a witness of $\homh$ with cardinality $\homh$. Thus all of the requirements for $\homh$ to be a very tame cardinal invariant are satisfied, and we are done.
\end{proof}

The following Lemma is essentially just~\cite[Theorem 6.1.16]{forcingidealized}, appropriately reformulated so that it is readily applicable in our context. The ``proof'' of the Lemma just consists of linking together the necessary results and definitions from~\cite{forcingidealized}, for the benefit of those readers that are interested in chasing the lemma all the way to first principles.

\begin{lem}\label{zapletal}
If $\mathfrak x$ is a very tame cardinal invariant such that the inequality $\mathfrak x<\mathfrak c$ holds in some forcing extension, then the iterated Sacks model satisfies that $\omega_1=\mathfrak x<\mathfrak c$.\footnote{In~\cite{zapletal}, a \textbf{tame} cardinal invariant is defined to be one whose definition is of the form
\begin{equation*}
\min\{|Y|\big|Y\subseteq\mathbb R\wedge\phi(Y)\wedge\psi(Y)\},
\end{equation*}
where the quantifiers of $\phi$ are restricted to the set $Y$ or to $\omega$, and $\psi(Y)$ is a sentence of the form $(\forall x\in\mathbb R)(\exists y\in Y)(\theta(x,y))$, with $\theta$ a formula whose quantifiers range over natural and real numbers only, without mentioning the set $Y$ (both formulas can have a real parameter). In that same paper~\cite[Theorem 0.2]{zapletal} it is established that, if there exists a proper class of measurable Woodin cardinals, then every tame cardinal invariant that can be made $<\mathfrak c$ in some forcing extension will be $<\mathfrak c$ in the iterated Sacks extension. So for the reader who feels entirely comfortable with large cardinal assumptions, much work can be saved by simply verifying that $\homh$ satisfies the definition of a tame cardinal invariant (much simpler than that of a very tame invariant) and then appealing to the aforementioned theorem to conclude that $\homh<\mathfrak c$ in the iterated Sacks model.
}
\end{lem}

\begin{proof}
In~\cite[Theorem 6.1.16]{forcingidealized} it is stated that whenever $I$ is an iterable $\Pi_1^1$ on $\Sigma_1^1$ $\sigma$-ideal on some Polish space, and $\mathfrak x$ is a very tame cardinal invariant, if $\mathfrak x<\cov^*(I)$ holds in some inner model of $\zfc$ containing all ordinals or in a generic extension of such an inner model, then $\omega_1=\mathfrak x<\cov^*(I)$ holds in every generic extension of every inner model of $\zfc$ containing all ordinals, provided that such an extension satisfies $\cpa(I)$. Since Sacks forcing is forcing equivalent to the quotient of all Borel subsets of $2^\omega$ modulo the ideal $I$ of countable sets, and $\cov^*(I)=\mathfrak c$ (as $\cov^*(I)$ is defined to be the least number of elements of $I$ required to cover an $I$-positive Borel set), it follows that, whenever a very tame cardinal invariant $\mathfrak x$ satisfies $\mathfrak x<\mathfrak c$ in some forcing extension, it will necessarily be the case that $\omega_1=\mathfrak x<\mathfrak c$ in the iterated Sacks model, as long as this model satisfies $\cpa(I)$ and the ideal $I$ is iterable $\Pi_1^1$ on $\Sigma_1^1$. By~\cite[Proposition 6.1.3]{forcingidealized}, under the hypothesized conditions, $\cpa(I)$ will hold in the generic extension obtained by CS iteration of the forcing $\mathbb P_I$, which is just the iterated Sacks model. Hence it only remains to see that the ideal $I$ of countable sets is iterable $\Pi_1^1$ on $\Sigma_1^1$. The latter property follows directly from~\cite[29.19, p. 231]{kechris} (see~\cite[Definition 3.8.1]{forcingidealized} for the definition of $\Pi_1^1$ on $\Sigma_1^1$). As for iterable, the reader is urged to look at~\cite[Definition 5.1.3]{forcingidealized} for the definition, which has three clauses, the first of which is satisfied by every $\Pi_1^1$ on $\Sigma_1^1$ ideal by~\cite[Proposition 2.1.22]{forcingidealized}. Now, that our ideal $I$ of countable sets satisfies the second clause (defined by~\cite[Definition 3.9.21]{forcingidealized}) follows from~\cite[Exercise 14.13, p. 88]{kechris}; and it moreover satisfies the third clause as well by~\cite[Proposition 2.2.2]{forcingidealized} and/or by~\cite[Example 2.2.5]{forcingidealized}.
\end{proof}

\begin{thm}\label{homhsmallinsacks}
In the iterated Sacks model, the equality $\homh=\omega_1$ holds.
\end{thm}

\begin{proof}
Refer to one of the models from~\cite[Section 4.3, Theorems 4.20 and 4.23]{tesis} (obtained by means of ``short'' finite support iterations of c.c.c. forcing notions), that satisfy $\covm<\mathfrak c$ together with the existence of union ultrafilters. It is pointed out in the construction of such models (as part of the argumentation to conclude that $\covm$ is small in these) that the relevant union ultrafilter has a basis of $\covm$-many sets of the form $\fu(X)$, and hence such basis constitutes a witness for $\homh<\mathfrak c$. Hence by Lemmas~\ref{homhistame} and~\ref{zapletal}, the conclusion is that $\homh=\omega_1$ in the iterated Sacks model.
\end{proof}

\begin{thm}\label{diamondholdsinsacks}
The parametrized diamond principle $\diamondsuit(\homh)$ holds in the iterated Sacks model.
\end{thm}

\begin{proof}
By Theorem~\ref{homhsmallinsacks}, the cardinal invariant $\homh$ equals $\omega_1$ in the iterated Sacks model. Since this model is obtained by a suitable forcing iteration, Theorem~\ref{fundamental} applies, and we conclude that $\diamondsuit(\homh)$ holds in this model.
\end{proof}

\section{Union ultrafilters from $\diamondsuit(\homh)$}

In the previous section we showed that $\homh=\omega_1$ in the iterated Sacks model. Thus, by Theorem~\ref{fundamental}, 
we have that $\diamondsuit(\homh)$ holds in this model, so our proof that there are union ultrafilters in the iterated Sacks model will be complete once we prove that such an existential statement follows from $\diamondsuit(\homh)$. 

For the convenience of the reader, we state again the combinatorial principle 
$\diamondsuit(\homh)$: this is the statement that for every function 
$F:2^{<\omega_1}\longrightarrow\wp([\omega]^{<\omega})$, there exists an 
$\omega_1$-sequence $g$ of block sequences such that, for every $f:\omega_1\longrightarrow2$, $F(f\upharpoonright\alpha)$ is almost $\fu$-reaped by $g(\alpha)$ (i.e. $\fu(g(\alpha))$ reaps 
$F(f\upharpoonright\alpha)$) for stationarily many $\alpha<\omega_1$.

In order to carry out our construction, we first need a couple of lemmas.

\begin{lem}\label{isomorphism}
Given any block sequence $X\subseteq[\omega]^{<\omega}$, there exists an 
isomorphism $\varphi:\fu(X)\longrightarrow[\omega]^{<\omega}\setminus\{\varnothing\}$ 
(by which we mean: a bijection such that, for any two
$x,y\in\fu(X)$, if $\max(x)<\min(y)$, then $\max(\varphi(x))<\min(\varphi(y))$, and 
furthermore $\varphi(x\cup y)=\varphi(x)\cup\varphi(y)$).
\end{lem}

\begin{proof}
Since $X$ is a block sequence, it is possible to enumerate its elements 
as $X=\{x_n\big|n<\omega\}$ in such a way that, for every $n<\omega$, 
$\max(x_n)<\min(x_{n+1})$.  Now just notice that the mapping 
$\bigcup_{n\in a}x_n\longmapsto a$ is as required.
\end{proof}

\begin{lem}\label{pseudoint}
Given any sequence $\langle X_n\big|n<\omega\rangle$ of block sequences 
satisfying $(\forall n<\omega)(X_{n+1}\subseteq^*\fu(X_n))$, it is possible to choose, 
in a Borel way, a block sequence $X$ such that, for all $n<\omega$, 
$X\subseteq^*\fu(X)$.
\end{lem}

\begin{proof}
We will recursively construct $X=\{x_n\big|n<\omega\}$ in such a way that, 
for all $n<\omega$, $\max(x_n)<\min(x_{n+1})$, and furthermore 
$\fu(\{x_k\big|k\geq n\})\subseteq X_n$. We do this as follows: since 
each $X_n$ is a 
block sequence, we can enumerate it as $X_n=\{x(n,k)\big|k<\omega\}$ in 
such a way that $\max(x(n,k))<\min(x(n,k+1))$. Let $x_0=x(0,0)$ and, 
assuming that we know $x_n$, let $k_n$ be (the least number) large enough that 
$\max(x_n)<\min(x(n+1,k_n))$ and that $(\forall k\geq k_n)(\forall i\leq n)(x(n+1,k)\in\fu(X_i))$. Then just let $x_{n+1}=x(n+1,k_n)$. 
Then our sequence $X$, by construction, is a block sequence and satisfies that $(\forall n<\omega)(\{x_k\big|k\geq n\}\subseteq\fu(X_n))$.
\end{proof}

These two lemmas encapsulate the key steps in the proof of the main theorem of this section. Before stating such theorem, we state a couple of definitions.

\begin{defin}
Let $u$ be a union ultrafilter.
\begin{itemize}
\item $u$ is said to be \textbf{stable} if for every countable collection $\{X_n\big|n<\omega\}$ of pairwise disjoint families satisfying $(\forall n<\omega)(\fu(X_n)\in u)$, there exists a pairwise disjoint $X$ such that $\fu(X)\in u$ and $(\forall n<\omega)(X\subseteq^*\fu(X))$.
\item $u$ is said to be an \textbf{ordered-union} ultrafilter if it satisfies the defining condition for union ultrafilters, where the ``pairwise disjoint'' requirement has been strengthened to ``ordered''. That is, $u$ is an ordered-union ultrafilter if $(\forall A\in u)(\exists X\text{ ordered})(A\supseteq\fu(X)\in u)$.
\end{itemize}
\end{defin}

Thus, the property of being a stable ordered-union ultrafilter is stronger than simply being a union ultrafilter. It has been shown by Blass and Hindman~\cite[Theorem 4]{blasshindman} that the existence of union ultrafilters that are not ordered is consistent with the $\zfc$ axioms (in fact, it follows from Martin's Axiom, even when restricted to countable forcing notions, i.e. $\covm=\mathfrak c$). On the other hand, it is still open whether one can have (consistently) union ultrafilters that are not stable. We see below that the principle $\diamondsuit(\homh)$ implies not only the existence of union ultrafilters, but in fact of stable ordered union ultrafilters.

\begin{thm}\label{homhimpliesunion}
 The combinatorial principle $\diamondsuit(\homh)$ implies the existence of 
 stable ordered-union ultrafilters of character $\omega_1$.
\end{thm}

\begin{proof}
 We let sequences in $2^{<\omega_1}$ code pairs $\langle A,\vec{X}\rangle$, 
 where $A\subseteq[\omega]^{<\omega}$, and 
 $\vec{X}=\langle X_\xi\big|\xi<\alpha\rangle$ is a sequence of 
 block sequences satisfying that for every $\xi<\beta<\alpha$, 
 $X_\beta\subseteq^*\fu(X_\xi)$. Also for every limit $\alpha<\omega_1$ we 
 fix an increasing cofinal sequence $\langle\alpha_n\big|n<\omega\rangle$. 
 Given a pair $\langle A,\vec{X}\rangle$, we define $F(\langle A,\vec{X}\rangle)$ 
 as follows: We first pick, using Lemma~\ref{pseudoint}, a block sequence 
 $Y$ such that for every $n<\omega$, 
 $Y\subseteq^*\fu(X_{\alpha_n})$. We then let 
 $\varphi_Y:\fu(Y)\longrightarrow[\omega]^{<\omega}\setminus\{\varnothing\}$ 
 be the isomorphism as in Lemma~\ref{isomorphism}, and finally define 
 $F(\langle A,\vec{X}\rangle)=\ran(\varphi_Y\upharpoonright(A\cap\fu(Y)))$.
 
  Now by our assumption that $\diamondsuit(\homh)$ holds, we can obtain 
 a $g:\omega_1\longrightarrow \{X\subseteq[\omega]^{<\omega}\big|X\text{ is an infinite block sequence}\}$ such that, for each $f:\omega_1\longrightarrow2$, 
 the set $\{\alpha<\omega_1\big|F(f\upharpoonright\alpha)\text{ is }\fu\text{-reaped by }\ g(\alpha)\}$ is stationary. 
 We use $g$ to recursively define our ultrafilter as follows. Suppose that we 
 already have $\langle X_\xi\big|\xi<\alpha\rangle$, satisfying that, 
 whenever $\xi<\eta<\alpha$, it is the case that $X_\eta\subseteq^*\fu(X_\xi)$.  Now 
grab the same cofinal sequence $\langle\alpha_n\big|n<\omega\rangle$ as in the 
construction of $F$, and let $Y$ be as in Lemma~\ref{pseudoint} for the 
sequence $\langle X_{\alpha_n}\big|n<\omega\rangle$. Let 
$\varphi_Y:\fu(Y)\longrightarrow[\omega]^{<\omega}\setminus\{\varnothing\}$ be the 
isomorphism from Lemma~\ref{isomorphism}, and define 
$X_\alpha=\varphi_Y^{-1}[g(\alpha)]\subseteq Y$. Then by construction, the final 
sequence $\langle X_\alpha\big|\alpha<\omega_1\rangle$ will be such that 
$\{\fu(X_\alpha)\big|\alpha<\omega_1\}$ generates a filter, which will be a 
union ultrafilter provided that we can show that it is actually an ultrafilter.

To see this, let $A\subseteq[\omega]^{<\omega}$ be arbitrary. Let 
$f:\omega_1\longrightarrow2$ be the function that codes the 
branch $\langle A,\langle X_\alpha\big|\alpha<\omega_1\rangle\rangle$ of our 
tree.  By the assumption on $g$, there will be an $\alpha$ (in fact, 
stationarily many of them, although we really only need one) such that $F(\langle A,\langle X_\xi\big|\xi<\alpha\rangle\rangle)$ is $\fu$-reaped by $g(\alpha)$. 
But recall that the former set is just $\ran(\varphi_Y\upharpoonright(A\cap\fu(Y)))$, where 
$Y$ was obtained from $\langle X_{\alpha_n}\big|n<\omega\rangle$ by means of 
Lemma~\ref{pseudoint}, and $\varphi_Y$ is the isomorphism from Lemma~\ref{isomorphism}. 
Since $\fu(g(\alpha))$ reaps this set, by pulling everything back via 
the isomorphism $\varphi_Y$, we get that $\varphi_Y^{-1}[\fu(g(\alpha))]=\fu(\varphi^{-1}[g(\alpha)])=\fu(X_\alpha)$ reaps $\varphi_Y^{-1}[\ran(\varphi_Y\upharpoonright(A\cap\fu(Y)))]=A\cap\fu(Y)$, and 
therefore also $A$. Thus our filter is actually an ultrafilter, and in fact a stable ordered union ultrafilter.
\end{proof}

As a corollary of the above, we obtain an answer to the question posed in~\cite[Question 4.11, p. 106]{tesis}, which was the main motivation behind this work.

\begin{cor}
In the iterated Sacks model, the statement ``there exists a stable ordered-union ultrafilter of character $\omega_1$ holds''.
\end{cor}

\begin{proof}
By Theorem~\ref{diamondholdsinsacks}, the combinatorial principle $\diamondsuit(\homh)$ holds in the iterated Sacks model, thus by Theorem~\ref{homhimpliesunion} this implies that there exists an ultrafilter as required in this model.
\end{proof}

\subsection*{Acknowledgements}
The first author was partially supported by Postdoctoral Fellowship 
number 275049 from the Consejo Nacional de Ciencia y Tecnolog\'{\i}a 
(CONACyT), Mexico. The second author gratefully acknowledges support form a CONACyT
grant 177758 and a PAPIIT grant IN 100317.

\end{document}